\documentclass[11pt,oneside]{amsart}
\usepackage[utf8]{inputenc}
\usepackage[margin=1in]{geometry}
\usepackage{amscd}
\usepackage{hyperref}
\usepackage{fontenc}
\usepackage{textcomp}
\usepackage{comment}
\theoremstyle{plain}
\usepackage{enumitem}
\usepackage{amssymb,amsmath,amsthm}
\usepackage{mathtools}
\usepackage{mathrsfs}
\usepackage{xcolor}
\usepackage[all]{xy}
\usepackage{cancel}
\usepackage{graphicx}
\usepackage{lscape}
\usepackage{cite}
\usepackage{mathabx}

\newtheorem*{thmnonnum}{Theorem}

\theoremstyle{main}

\theoremstyle{main}

\newtheorem{rema}{Remarque}
\newtheorem*{remasnonnum}{Remark}

\newtheorem*{remer}{Acknowledgements}

\newtheorem*{coropasnum}{Corollary}
\newtheorem*{propnonnum}{Proposition}

\newtheorem*{anne}{Complement}

\newcommand\nn{\mathbb{N}}
\newcommand\zz{\mathbb{Z}}

\newcommand\sss{\mathbb{S}}

\newcommand\ungra{\mathbf{1}}

\newcommand\fra[2]{\displaystyle\frac{#1}{#2}}

\newcommand\cali[1]{\mathcal{#1}}
\newcommand*\diff{\mathop{}\!\mathrm{d}}

\DeclareMathOperator{\Var}{Var}

\DeclareMathOperator{\MSE}{MSE}

\title{On the optimality of the Monte-Carlo estimator}

\author{Antoine Pinochet Lobos}

\begin{document}

\maketitle

\renewcommand{\abstractname}{Abstract}
\begin{abstract}
We prove that on an atomless probability space, the worst-case mean squared error of the Monte-Carlo estimator is minimal if the random points are chosen independently.
\end{abstract}

\section{Introdution and statement of the results}

Let $(X,\mu)$ be a probability space. We are interested in the following general question: if $f$ is a measurable, real or complex-valued function on $X$, how can we efficiently compute the integral $\int_X f\diff \mu$ ? The famous \textit{Monte-Carlo method} is a solution to this problem: just choose an integer $n$ big enough, and draw $Z_1,\cdots,Z_n$ independent $X$-valued random variables (that is, \textit{random points}) of law $\mu$, and form the mean $\frac{1}{n} \sum^n_{i=1} f(Z_i)$, called the \textit{Monte-Carlo estimator}.

We measure the quality of this method by computing what we call the \textit{mean squared error}: we have the well-known equality, valid for all $n\in \nn^*$ and $f \in L^2(X,\mu)$, \[\Var\left(\frac{1}{n}\sum^n_{i=1} f(Z_i) - \int_X f\diff\mu\right) = \frac{1}{n}\left\Vert f - \int_X f \diff \mu \right\Vert^2_{L^2(X,\mu)}\]and we obtain the following equality, concerning the \textit{worst-case mean squared error}:

\[\sup_{\substack{f \in L^2(X,\mu) \\ \Vert f \Vert_2 = 1}} \Var\left(\frac{1}{n}\sum^n_{i=1} f(Z_i) - \int_X f \diff \mu\right) = \frac{1}{n}.\]

In this paper, we study the question of measuring the worst-case mean squared error, in the general situation where the points $Z_i$ are not supposed independent, and we prove the following theorem and its corollary.

\begin{thmnonnum} Let $(X,\mu)$ be a probability space, let $N,n \in \nn^*$, and $Z := (Z_1,\cdots,Z_n)$ an $n$-tuple of random points on $X$ such that for all $i$, the law of $Z_i$ is $\mu$. We do not assume that the $Z_i$'s are independent. Furthermore, we assume that $X$ can be partitioned in $N$ measurable subsets of equal measure.

We then have \[\sup_{\substack{f \in L^2(X,\mu) \\ \Vert f \Vert_2 = 1}} \Var\left(\frac{1}{n}\sum^n_{i=1} f(Z_i) - \int_X f \diff \mu\right) \geq \frac{1}{n}\left(1 - \fra{n-1}{N-1}\right).\]
\end{thmnonnum}

\begin{coropasnum} Let $(X,\mu)$ be an atomless probability space, $n \in \nn^*$, and $Z := (Z_1,\cdots,Z_n)$ an $n$-tuple of random points on $X$ such that for all $i$, the law of $Z_i$ is $\mu$. We do not assume that the $Z_i$'s are independent.

We then have\[\sup_{\substack{f \in L^2(X,\mu) \\ \Vert f \Vert_2 = 1}} \Var\left(\frac{1}{n}\sum^n_{i=1} f(Z_i) - \int_X f \diff \mu\right) \geq \frac{1}{n}.\]
\end{coropasnum}

\begin{remasnonnum} As we shall see in the paper, in the case where $X := \{1,\cdots,N\}$ and $\mu$ is the uniform measure on $X$, the inequality of the theorem is an equality when the law of $Z$ is the uniform measure on the set of $n$-tuples of points in $X$ such that the coordinates are pairwise different. This random $n$-tuple is then, in the sense of the worst-case mean squared error, than an independent $n$-tuple.

As we saw before, the inequality in the corollary is an equality if the $Z_i$'s are independent. We don't know if this condition is necessary. It is, to our opinion, worth knowing that in \cite{LPS}, the authors build, for all prime $p$ such that $p \equiv 1 [4]$, a $(p+1)$-tuple $Z$ of uniform random points on the $2$-sphere which are not independent, and prove that its worst-case mean squared error is $\frac{4p}{(p+1)^2}$, which is approximately $4$ times the lower bound in the corollary. In the article \cite{PinPit}, it is shown that their construction is optimal, in a broad framework.
\end{remasnonnum}

We confess our astonishment of having found no trace of these statements, which answer a question that we find both natural and general, and in an elementary way.

\begin{remer} We woud like to thank Sébastien Darses, Thibault Espinasse, Alexandre Gaudillière, Pierre Mathieu, Clothilde Melot, Pierre Pudlo and more particularly Christophe Pittet for the conversations that we had about the questions studied in this paper, and for their encouragement.
\end{remer}

\section{Proofs}

To alleviate the presentation, we use the following notation: we consider the numbers \[\MSE_Z(f) := \Var\left(\frac{1}{n}\sum^n_{i=1} f(Z_i) - \int_X f \diff \mu\right)\] et 
\[\MSE(Z) := \sup_{\substack{f \in L^2(X,\mu) \\ \Vert f \Vert_2 = 1}} \MSE_Z(f).\]

First of all, if $f \in L^2(X,\mu)$, we notice that $\MSE_Z(f) = \MSE(Z)\left(f-\int_X f \diff \mu\right)$. Consequently, $\MSE(Z)$ is also the $\sup$ of the $\MSE_Z(f)$ for $f$ of norm $1$ zero integral. 

Let $f \in L^2(X,\mu)$, of norm $1$ and zero integral. We have that \[\begin{array}{rcl}
\MSE_Z(f) &= &\displaystyle\mathbb{E}\left[\left(\frac{1}{n}\sum^n_{i=1} f(Z_i)\right)^2\right]\\
&= &\displaystyle\mathbb{E}\left[\frac{1}{n^2}\sum^n_{i=1} f(Z_i)^2 + \frac{1}{n^2} \sum_{i \neq j} f(Z_i)f(Z_j)\right]\\
&= &\displaystyle\frac{1}{n^2} \sum^n_{i=1}\mathbb{E}[f(Z_i)^2] + \frac{1}{n^2} \sum_{i\neq j} \mathbb{E}[f(Z_i)f(Z_j)]\\
&= &\displaystyle \frac{1}{n} + \frac{1}{n^2}\sum_{i\neq j} \mathbb{E}[f(Z_i)f(Z_j)]\\
\end{array}\]

and we recover the fact recalled above: if the $Z_i$'s are pairwise independent, and if $f$ is of norm $1$ of zero integral, $\MSE_Z(f) = \frac{1}{n}$.

Let us prove the theorem.

\begin{proof}[Proof of the theorem] Let $X_1,...,X_N$ be measurable subsets that partition $X$, all of measure $\frac{1}{N}$, with $N \geq 2$. Let us denote, for $p \in \{1,...,N\}$, $\mu_p := \mu(X_p)$. For every $(p,q) \in \{1,...,N\}^2$, we set \[f_{p,q} := \sqrt{\frac{N}{2}} \ungra_{X_p} - \sqrt{\frac{N}{2}} \ungra_{X_q}.\]Moreover, we will denote, for $k \in \{1,...,N\}$, $f_{p,q}(X_k)$ the value that $f_{p,q}$ takes on $X_k$ - this abuse of notation is harmless because $f_{p,q}$ is constant on the $X_i$'s.

$f_{p,q}$ is visibly of zero integral, and if $p\neq q$, its norm is $1$.

We will prove that there are different $p,q \in \{1,...,N\}$ such that $\MSE_Z(f_{p,q}) \geq \frac{1}{n}\left(1 - \frac{n-1}{N-1}\right)$.

Let $p,q \in \{1,...,N\}$.
We have that \[\begin{array}{rcll}
\MSE_Z(f_{p,q}) &= &\displaystyle\frac{1}{n} &\displaystyle + \ \frac{1}{n^2}\sum_{i\neq j} \mathbb{E}\left[f_{p,q}(Z_i)f_{p,q}(Z_j)\right]\\
&= &\displaystyle \frac{1}{n} &\displaystyle +\ \frac{1}{n^2} \sum_{i \neq j}\left(\sum_k \mathbb{P}(Z_i \in X_k \ et \ Z_j \in X_k)f_{p,q}(X_k)^2\right. \\
&&&\displaystyle + \left.\sum_{l \neq m} \mathbb{P}(Z_i \in X_l \ et \ Z_j \in X_m)f_{p,q}(X_l)f_{p,q}(X_m)\right)\\
&= &\displaystyle \frac{1}{n} &\displaystyle +\ \frac{1}{n^2}\frac{N}{2} \sum_{i \neq j}\left(\mathbb{P}(Z_i \in X_p \ et \ Z_j \in X_p)\right. \\
&&&\displaystyle +\ \mathbb{P}(Z_i \in X_q \ et \ Z_j \in X_q)\\
&&&\displaystyle -\ \mathbb{P}(Z_i \in X_p \ et \ Z_j \in X_q)\\
&&&\displaystyle -\left.\ \mathbb{P}(Z_i \in X_q \ et \ Z_j \in X_p)\right).\\
\end{array}\]from which we deduce the inequality \[\MSE_Z(f_{p,q}) \geq \frac{1}{n} - \frac{1}{n^2}\frac{N}{2}\left(\sum_{i \neq j} \mathbb{P}(Z_i \in X_p \ et \ Z_j \in X_q) + \mathbb{P}(Z_i \in X_q \ et \ Z_j \in X_p)\right).\]Let us denote \[\theta_{p,q} := \sum_{i \neq j} \mathbb{P}(Z_i \in X_p \ et \ Z_j \in X_q) + \mathbb{P}(Z_i \in X_q \ et \ Z_j \in X_p).\]

Let us compute: \[\begin{array}{rcl}
\displaystyle\sum_{p \neq q} \theta_{p,q} &= &2\sum_{p\neq q} \sum_{i \neq j} \mathbb{P}(Z_i \in X_p\ et\ Z_j \in X_q)\\
&= &2\sum_{i \neq j}\sum_{p \neq q} \mathbb{P}(Z_i \in X_p \ et \ Z_j \in X_q)\\
&= &2\sum_{i \neq j} \mathbb{P}(Z_i \ \mbox{et}\ Z_j\ \mbox{ne sont pas dans le même morceau de la partition})\\
&\leq &2n(n-1).\\
\end{array}\]

Now, since this sum of $N(N-1)$ numbers is lower or equal than $2n(n-1)$, then one of the terms must be lower or equal than $2\frac{n(n-1)}{N(N-1)}$. For a couple $(p,q)$ such that $\theta_{p,q} \leq 2\frac{n(n-1)}{N(N-1)}$, we then have \[\begin{array}{rcl}
\MSE_Z(f_{p,q}) &\geq &\displaystyle\frac{1}{n} - \frac{1}{n^2}\frac{N}{2}2\frac{n(n-1)}{N(N-1)}\\
&= &\displaystyle \frac{1}{n} - \frac{n-1}{n(N-1)}\\
&= &\displaystyle \frac{1}{n}\left(1 - \frac{n-1}{N-1}\right).\\
\end{array}\]
\end{proof}

Here's an example where the inequality is an equality.

\begin{propnonnum} If $X := \{1,\cdots,N\}$, if $\mu$ is the uniform probability on $X$, if $n \leq N$, and if the law of $Z$ is the uniform measure on the set of $n$-tuples of points in $X$ which coordinates are pairwise different, then the inequality in the theorem is an equality, that is, \[\MSE_Z = \fra{1}{n}\left(1 - \fra{n-1}{N-1}\right).\]
\end{propnonnum}

\begin{proof}[Proof] Let $\pi$ be the measure on $X^n$ defined by \[\pi := \fra{(N-n)!}{N!}\sum_{\substack{i_1,\cdots,i_n \in X\\ \forall j\neq k, \\ i_j \neq i_k}}\delta_{(i_1,\cdots,i_n)},\]where $\delta$ is the notation for a Dirac measure. In words, $\pi$ is the uniform measure on the set of $n$-tuples of points in $X$ which coordinates are pairwise different. Let $Z = (Z_1,\cdots, Z_n)$ be an $n$-tuple of law $\pi$ (we then have, for all $i$, that $Z_i$ is uniform on $X$).

Let $f \in L^2(X,\mu)$ be of norm $1$, and such that $\int f \diff\mu = 0$. Let us compute: 

\[\begin{array}{rcl}
\sum_{l\neq m} \mathbb{E}[f(Z_l)f(Z_m)] &= &\displaystyle\sum_{l\neq m} \mathbb{E}\left[\sum_{\substack{A\subset X\\ \vert A \vert = n}} \sum_{\substack{i_1,\cdots,i_n \in A\\ \forall j\neq k\\ i_j \neq i_k}} \ungra_{\{Z_1 = i_1,\cdots,Z_n = i_n\}} f(Z_l)f(Z_m)\right]\\
&= &\displaystyle \sum_{l\neq m}\sum_{\substack{A\subset X\\ \vert A \vert = n}} \sum_{\substack{i_1,\cdots,i_n \in A\\ \forall j\neq k\\ i_j \neq i_k}}  \mathbb{P}\left[Z_1 = i_1,\cdots,Z_n = i_n\right] f(i_l)f(i_m)\\
&= &\displaystyle \dbinom{N}{n}^{-1}\sum_{\substack{A\subset X\\ \vert A \vert = n}}  \sum_{\substack{p,q \in A\\ p\neq q}}  f(p)f(q) \\
&= &\displaystyle \dbinom{N-2}{n-2}\dbinom{N}{n}^{-1} \sum_{\substack{p,q \in X\\ p\neq q}}  f(p)f(q) \\
&= &\displaystyle\frac{n(n-1)}{N(N-1)} \sum_{p \in X} f(p)\sum_{\substack{q \in X \\ q\neq p}} f(q)\\
&= &\displaystyle \frac{n(n-1)}{N-1}\Vert f \Vert^2_2\\
&= &\displaystyle \frac{n(n-1)}{N-1}.\\
\end{array}\]

We therefore have \[\MSE_Z(f) = \fra{1}{n}\left(1 - \frac{n-1}{N-1} \right).\]

\end{proof}

Let us prove the corollary.

\begin{proof}[Proof of the corollary] We will prove that for all $\epsilon > 0$, we have that $\MSE(Z) \geq \frac{1}{n} - \epsilon$, which is enough. According to a theorem of Sierpiński \cite{Sierpinski}, every atomless probability space is such that for every $a \in [0,1]$, there is a measurable subset of $X$ of measure $a$. From this, it is easy, for all arbitrarily big $N$, to partition $X$ in $N$ of measurable subsets of equal measure. If we choose $N$ such that $\frac{1}{n}\left(1 - \frac{n-1}{N-1}\right) \geq \frac{1}{n} - \epsilon$, which is obviously possible, then according to the theorem, it is possible to find $f$ of norm $1$, zero integral, such that $\MSE_Z(f) \geq \frac{1}{n} - \epsilon$.
\end{proof}

For the sake of completeness, we add a simple proof of Sierpiński's theorem.

\begin{anne}[Sierpiński's theorem on atomless probability spaces]

If $(X,\mathscr{B},\mu)$ is an atomless probability space, then for every measurable $A \subset X$, there exists $\phi : [0,\mu(A)] \rightarrow \mathscr{B}$ non-decreasing, such that $\forall t \in [0,\mu(A)],\quad \mu(\phi(t)) = t$.
\end{anne}

\begin{proof} The hypothesis of $X$ being atomless means that for every measurable $B \subset X$ such that $\mu(B) > 0$, there exists a measurable $C \subset B$ such that $0<\mu(C) < \mu(B)$.

Let $A$ be a measurable subset of $X$, such that $\mu(A) > 0$ (if $\mu(A) = 0$, it is enough to define $\phi(0) := A$). 
By applying Zorn's lemma, we obtain a $\phi : I \rightarrow \mathscr{B}$ where $I$ is a subset of $[0,\mu(A)]$, $\phi$ is non-decreasing, such that $\forall i \in I$, $\mu(\phi(i)) = i$, and such that $\phi$ has no strict extension that satisfies these properties. Let us show that $I$ equals $[0,\mu(A)]$.

On the one hand, $I$ is closed. Indeed, let $(x_n)_{n \in \nn}$ be a sequence of elements in $I$ that converges to some $x$. Let us show that $x \in I$. We can assume, up to extracting a subsequence, that $(x_n)_n$ is monotonous. If $x \not \in I$, let us define $\tilde{\phi} := I \cup \{x\} \rightarrow \mathscr{B}$ that extends $\phi$ by defining $\tilde{\phi}(x) := \bigcap_n \phi(x_n)$ if $(x_n)_n$ is non-increasing, and $\tilde{\phi}(x) := \bigcup_n \phi(x_n)$ if $(x_n)_n$ non-decreasing. According to $\mu$'s continuity properties, $\mu(\phi(x)) = \lim_n x_n = x$, and according to the monotony properties of $\mu$, $\tilde{\phi}$ is non-decreasing. $\tilde{\phi}$ is therefore a strict extension of $\phi$ that verifies the same properties. This is a contradiction. So $x \in I$, and therefore, $I$ is closed.

On the other hand, $I$ verifies $\forall a,b \in I,\ a < b \Rightarrow \left(\exists c \in I, \ a < c < b\right)$ (we say that $I$ is \textit{order-dense}). Indeed, if there are $a,b \in I$ such that $a < b$ and $]a,b[ \cap I = \emptyset$, then let us use the hypothesis that $X$ is atomless, which provides a measurable $C \subset \phi(b) \setminus \phi(a)$ such that $0 < \mu(C) < b - a$. Let us then define $\tilde{\phi} : I \cup \{a + \mu(C)\}$ that extends $\phi$ by defining $\tilde{\phi}(a + \mu(C)) := \phi(a) \cup C$. Then $\mu(\tilde{\phi})(a + \mu(C)) = \mu(\phi(a) \cup C) = a + \mu(C)$. According to $\mu$'s monotony properties, $\tilde{\phi}$ est non-decreasing. $\tilde{\phi}$ is then a strict extension of $\phi$ that verifies the same properties. This is a contradiction. Therefore, $I$ is order-dense.

So $I$ is closed and order-dense. Therefore, $I = [0,\mu(A)]$.
\end{proof}

\bibliographystyle{alpha}
\bibliography{bibliomontecarlo}

\end{document}